\documentclass{amsart}   	
\usepackage{geometry}                		
\usepackage{graphicx}				
\usepackage{amsmath}								
\usepackage{amssymb}
\usepackage{amsthm}
\usepackage{fancyhdr}

\newtheorem{theorem}{Theorem}[section]

\newtheorem{definition}{Definition}[section]
\newtheorem{lemma}[theorem]{Lemma}
\newtheorem{prop}{Proposition}[section]

\newtheorem{cor}[theorem]{Corollary}
\newtheorem{example}{Example}[section]

\title{Gradient K\"ahler-Ricci solitons with nonnegative orthogonal bisectional curvature}
\author{Shijin Zhang}
\address{School of Mathematics and Systems Science, Beihang University, Beijing 100191, P.R. China}
\email{shijinzhang@buaa.edu.cn}
\date{}							

\begin{document}
\maketitle
\begin{abstract}
\noindent
In this paper, we prove that any complete shrinking gradient K\"ahler-Ricci solitons with positive orthogonal bisectional curvature must be compact. We also obtain a classification of the complete shrinking gradient K\"ahler-Ricci solitons with nonnegative orthogonal bisectional curvature.
\end{abstract}

\section{Introduction}
A gradient Ricci soliton is a self-similar solution to the Ricci
flow which flows by diffeomorphism  and homothety.  The study of solitons has become
increasingly important in both the study of the Ricci
flow and metric measure theory. In Perelman's proof of Poincar\'e conjecture , one  issue he needs to prove is that three dimensional shrinking  gradient Ricci soliton with positive sectional curvature is compact. It is a natural question to ask whether this   holds in higher dimension. Recently, Munteaun and Wang \cite{Munteanu-Wang} proved the following.
\begin{theorem}[Munteanu-Wang \cite{Munteanu-Wang}]
Any complete shrinking gradient Ricci soliton with nonnegative sectional curvature and positive Ricci curvature is compact.
\end{theorem}
In this paper, we consider the gradient K\"ahler-Ricci soliton, namely a triple $(M^n, g, f)$ associated with a K\"ahler manifold $(M,g)$ such that
\begin{equation}\label{solitonequation}
R_{i\overline{j}}+\nabla_{i}\nabla_{\overline{j}}f=\beta g_{i\overline{j}},\quad\quad {\rm and}\quad\quad \nabla_{i}\nabla_{j}f=0.
\end{equation}
for some constant $\beta\in {\mathbb R}$. It is called shrinking, steady or expanding, if $\beta>0$, $\beta=0$ or $\beta<0$ respectively.

In fact, gradient Ricci solitons are special solutions of Ricci flow. Let $\tau(t):=1-2\beta t>0$ and $\varphi(t):M^{n}\rightarrow M^{n}$ is the $1$-parameter family of diffeomorphisms generated by $X(t):=\frac{1}{\tau(t)}\nabla_{g}f$, that is,
\[\frac{\partial}{\partial t}\varphi(t)(x)=\frac{1}{\tau(t)}\nabla_{g}f(\varphi(t)(x)).\]
Let $g(t)=\tau(t)\varphi(t)^{*}g$. Then $g(t)$ is a solution of Ricci flow:
\begin{equation}\label{RF}
\frac{\partial}{\partial t}g(t)=-2{\rm Ric}(g(t))
\end{equation}
with $g(0)=g$.

Ni \cite{Ni} proved that any complete shrinking gradient K\"ahler Ricci soliton with positive bisectional curvature is compact. Recently, the author and Wu \cite{Wu-Zhang} using Munteanu-Wang's argument to provide an alternative proof of Ni's theorem.
\begin{theorem}[Ni, Wu-Zhang]
Any complete shrinking gradient K\"ahler-Ricci soliton with nonnegative bisectional curvature and positive Ricci curvature is compact.
\end{theorem}

In this short paper we will consider another type of curvature, orthogonal holomorphic bisectional curvature, which was introduced by Cao and Hamilton \cite{Cao-Hamilton}.
\begin{definition}
A K\"ahler manifold $(M^{n}, g) (n\geq 2)$ is said to have nonnegative (or positive) orthogonal bisectional curvature if for any orthonormal basis $\{e_{i}\}$, we have
\begin{equation}
R(e_{i},\overline{e_{i}}, e_{j},\overline{e_{j}})\geq 0 \quad( {\rm or} >0)
\end{equation}
for any $i\neq j.$
\end{definition}
Cao and Hamilton \cite{Cao-Hamilton} observed that the nonnegativity of the orthogonal holomorphic bisectional curvature is preserved under the K\"ahler-Ricci flow on the compact K\"ahler manifold.  Chen \cite{Chen} generalized the Frankel conjecture with positive orthogonal bisectional curvature but under the additional condition $c_{1}(M)>0$.  Gu and Zhang \cite{Gu-Zhang} proved that the positive orthogonal bisectional curvature implies $c_{1}(M)>0$, hence they proved the generalized Frankel conjecture with positive orthogonal bisectional curvature. Gu and Zhang also gave a complete classification of compact manifolds with nonnegative orthogonal bisectional curvature, see Theorem 1.3 in \cite{Gu-Zhang}.

In this paper,  using the argument of Munteanu and Wang \cite{Munteanu-Wang}, we prove that there is no noncompact complete shrinking gradient K\"ahler-Ricci soliton with nonnegative orthogonal bisectional curvature and positive Ricci curvature.
\begin{theorem}\label{MainThm1}
Let $(M^{n},g, f)$ be a complete shrinking gradient K\"ahler-Ricci soliton with nonnegative orthogonal bisectional curvature and positive Ricci curvature, then $M$ is compact.
\end{theorem}

For general K\"ahler manifold with nonnegative orthogonal bisectional curvature can not imply its Ricci curvature is nonnegative. There is an example from Gu-Zhang \cite{Gu-Zhang}.
\begin{example}
Let $(M,g)=(\Sigma, g_{1})\times (\mathbb{P}^{m}, g_{2})$, where $\Sigma$ is a Riemann surface with Gauss curvature $\kappa(\Sigma)\geq -4$ and $ \min(\kappa(\Sigma))=-4$ and $g_{2}$ is the standard Fubini-Study metric such that the sectional curvature of $\mathbb{P}^{m}$ is $1$. Then the orthogonal bisectional curvature of $(M,g)$ is nonnegative, but the Ricci curvature of $(M,g)$ is not nonnegative.
\end{example}
In fact, Gu and Zhang also showed that the isotropic curvature of $(M,g)$ is not nonnegative.

But for the complete steady or shrinking gradient K\"ahler-Ricci soliton with nonnegative orthogonal bisectional curvature, we show that the Ricci curvature is nonnegative. We have the following proposition.
\begin{prop}\label{RicNonnegative}
Let $(M^{n}, g, f,\beta)$ be a complete gradient K\"ahler-Ricci solitons with nonnegative orthogonal bisectional curvature. Then we have\\
{\rm (i)} If $\beta\geq 0$, then ${\rm Ric}(g)\geq 0$. Furthermore, if the orthogonal bisectional curvature is positive, then ${\rm Ric}(g)> 0$.\\
{\rm (ii)} If $\beta<0$, then ${\rm Ric}(g)\geq \beta g.$
\end{prop}

Using Proposition \ref{RicNonnegative}, Theorem \ref{RicNonnegative} and de Rham decomposition theorem, we obtain a classification of complete shrinking gradient K\"ahler-Ricci solitons with nonnegative orthogonal holomorphic bisectional curvature.
\begin{theorem}\label{MainThm2}
Let $(M^n, g, f)$ be a complete shrinking gradient K\"ahler-Ricci soliton with nonnegative orthogonal bisectional curvature. Then we have\\
{\rm(i)} If the orthogonal bisectional curvature of $M$ is positive then $M$ must be isometric-biholomorphic to ${\mathbb C}P^{n}$;\\
{\rm(ii)} If $M$ has nonnegative orthogonal bisectional curvature then the universal cover $\tilde{M}$ splits as $\tilde{M}=N_{1}\times N_{2}\times\cdots\times N_{l}\times {\mathbb C}^{k}$ isometric-biholomorphically, where $N_{i}$ are compact irreducible Hermitian Symmetric Spaces.
\end{theorem}

In section 2, we recall some preliminaries for gradient Ricci solitons. In section 3, we prove Proposition \ref{RicNonnegative}. In section 4, we prove Theorem \ref{MainThm1} and Theorem \ref{MainThm2}.

\section{Preliminaries}
In this section, we recall some famous formulae  for gradient K\"ahler-Ricci solitons. Let $(M^n, g, f,\beta) (\beta\in \mathbb{R})$ be a gradient K\"ahler-Ricci solitons, i.e.,
\begin{equation}
R_{i\overline{j}}+\nabla_{i}\nabla_{\overline{j}}f=\beta g_{i\overline{j}},\quad\quad {\rm and}\quad\quad \nabla_{i}\nabla_{j}f=0.
\end{equation}
Then the following formulae are well known.
\begin{lemma}
\quad
\begin{itemize}
\item[(1)] $R+\Delta f=\beta n$;
\item[(2)] $\nabla_{i} R=R_{i\overline{k}}\nabla_{k}f$;
\item[(3)] $R+|\nabla f|^{2}-\beta f={\rm constant}$;
\item[(4)] $\Delta_{f}R_{i\overline{j}}=\beta R_{i\overline{j}}-R_{i\overline{j}l\overline{k}}R_{k\overline{l}};$
\item[(5)] $\Delta_{f}R=\beta R-|{\rm Ric}|^{2}.$
\end{itemize}
Here $|\nabla f|^{2}=g^{i\overline{j}}\nabla_{i}f\nabla_{\overline{j}}f$ and $\Delta_{f}R_{i\overline{j}}=\Delta R_{i\overline{j}}-g^{k\overline{l}}\nabla_{k}f\nabla_{\overline{l}}R_{i\overline{j}}.$
\end{lemma}
\begin{proof}
The formulas are well known for gradient Ricci solitons, see Chapter 1 in \cite{CLN}.
\end{proof}
Cao and Zhou \cite{Cao-Zhou}, also see Fang-Man-Zhang \cite{Fang-Man-Zhang} or Haslhofer-M\"{u}ller \cite{Haslhofer-Muller}, proved the following estimate for the potential function $f$ in the case $\beta=1.$
\begin{lemma}[Cao-Zhou]
For $\beta=1$. Let $p$ be a point such that $\nabla f(p)=0$, $d(x):=d(x,p)$ denotes the distance function from $x$ to $p$. Then there exist uniform constants $c_1$ and $c_2$ such that
\begin{equation}
\frac{1}{4}(d(x)-c_1)^{2}\leq f(x) \leq \frac{1}{4}(d(x)+c_2)^{2}.
\end{equation}
\end{lemma}

We also recall an important Perelman's lemma, which will be used in the proof of Proposition \ref{RicNonnegative}.
\begin{lemma}[Perelman]\label{Perelman8.3}
Let $(M^{n},g,f,\beta)$ be a complete gradient Ricci soliton.
Fix $o\in M^{n}$, and define $r(x)\doteqdot d(o,p)$. Let $m=2n$. Suppose ${\rm Ric}(g)\leq (m-1)K$ on $B(o,r_{0})$, for some positive
numbers $r_{0}$ and $K$. Then for any point $x$, outside
$B(o,r_{0})$, we have
\begin{center}
$(\Delta r-<\nabla f,\nabla r>)(x)\leq -<\nabla f,\nabla
r>(o)-\frac{\beta}{2} r(x)+\frac{m-1}{2}\{\frac{2}{3}Kr_{0}+r_{0}^{-1}\}$.
\end{center}
\end{lemma}
The above lemma follows from an idea of Perelman;
see Lemma 8.3 in \cite{Perelman} and its antecedent in \S17 on 'Bounds on
changing distances', in \cite{Hamilton}. For the detailed proof, also see Proposition 2.2 in \cite{Zhz} or \cite{Chow}.

\section{Proof of Proposition \ref{RicNonnegative}}
In this section, we prove the Proposition \ref{RicNonnegative}, which will be used to prove the Theorem \ref{MainThm2}.
\begin{prop}
Let $(M^{n}, g, f,\beta)$ be a complete gradient K\"ahler-Ricci solitons with nonnegative orthogonal bisectional curvature. Then we have\\
{\rm (i)} If $\beta\geq 0$, then ${\rm Ric}(g)\geq 0$. Furthermore, if the orthogonal bisectional curvature is positive, then ${\rm Ric}(g)> 0$.\\
{\rm (ii)} If $\beta<0$, then ${\rm Ric}(g)\geq \beta g.$
\end{prop}
\begin{proof}
We only consider the case of noncompact manifold. Denote $\lambda(x)$ as the minimal eigenvalue of the Ricci curvature at $x$, suppose $v$ is the eigenvector corresponding to $\lambda(x)$, then
\begin{eqnarray*}
R_{i\overline{j}k\overline{l}}R_{\overline{k}l}v^i v^{\overline{j}}= R(v,\overline{v},\frac{\partial}{\partial z^{k}},\frac{\partial}{\partial z^{\overline{l}}})R_{\overline{k}l}
\end{eqnarray*}
Diagonalizing ${\rm Ric}$ at $x$ so that $R_{\overline{k}l}=\lambda_k \delta_{kl}$, and we may assume that $\lambda(x)=\lambda_{1}\leq \lambda_{2}\leq \cdots \leq \lambda_{n}$. Let $v(x)=\frac{\partial}{\partial z^{1}}$. We can write following as
\begin{eqnarray}\label{8}
\begin{aligned}
R_{i\overline{j}k\overline{l}}R_{\overline{k}l}v^i v^{\overline{j}}&=R(v,\overline{v},\frac{\partial}{\partial z^{k}},\frac{\partial}{\partial z^{\overline{k}}})\lambda_k\\
&=R(v,\overline{v}, v, \overline{v})\lambda_{1}+\sum_{1\neq j}R_{1\overline{1}j\overline{j}}\lambda_{j}\\
&=(\lambda_{1}-\sum_{1\neq j}R_{1\overline{1}j\overline{j}})\lambda_{1}+\sum_{1\neq j}R_{1\overline{1}j\overline{j}}\lambda_{j}\\
&=\lambda_{1}^{2}+\sum_{1\neq j}R_{1\overline{1}j\overline{j}}(\lambda_{j}-\lambda_{1})
\end{aligned}
\end{eqnarray}
Since the orthogonal bisectional curvature is nonnegative, we know at $x$
\begin{equation}\label{nonnegativity}
R_{i\overline{j}k\overline{l}}R_{\overline{k}l}v^i v^{\overline{j}}\geq \lambda(x)^{2}.
\end{equation}

 Fix $o\in M^{n}$ and fix a large number $b$. Let
\begin{center}
$\eta:[0,\infty)\rightarrow [0,1]$
\end{center}
be a $\mathcal{C}^{\infty}$ nonincreasing cutoff function with
$\eta(u)=1$ for $u\in [0,1]$ and $\eta(u)=0$ for $u\in
[1+b,\infty)$. Define $\Phi : M\rightarrow \mathbb{R}$ by
\begin{center}
$\Phi(x)=\eta(\frac{r(x)}{c})\lambda(x)$
\end{center}
for $c\in (0,\infty)$. Later we shall take $c\rightarrow \infty$.

Suppose $x_{0}\in M$ is such that
\begin{equation}
\Phi(x_{0})=\min_{M}\Phi<0.
\end{equation}
Then $\lambda(x_{0})<0$. Suppose $v\in T_x^{1,0}M$ is the unit eigenvector corresponding to $\lambda(x_{0})$,
  taking parallel translation of $v$  along any unit speed geodesic starting from $x_{0}$, then in a small neighborhood $B(x, \delta)$, we get a
  smooth vector field $V(x)$ with $V(x_{0})=v$. Define $h(x)={\rm Ric}(x)(V(x), \overline{V(x)})$, then $h(x)\geq \lambda(x)$ for $x\in B(x_{0},\delta)$ and $h(x_{0})=\lambda(x_{0})$, moreover, using $\nabla V(x)=0$ and since $R_{i\overline{j}k\overline{l}}R_{\overline{k}l}V^i V^{\overline{j}}(x)$ is continuous and $h(x)$ is continuous, and $R_{i\overline{j}k\overline{l}}R_{\overline{k}l}V^i V^{\overline{j}}(x_{0})\geq \lambda^{2}(x_{0})=h^{2}(x_{0})>0$ follows from (\ref{nonnegativity}), then for any $\epsilon>0$ there exists $0<\delta_{1}=\delta_{1}(\epsilon)<\delta$, such that $R_{i\overline{j}k\overline{l}}R_{\overline{k}l}V^i V^{\overline{j}}(x)\geq h(x)^{2}/(1+\epsilon)$ for any $x\in B(x_{0},\delta_{1})$. Hence for any $x\in B(x_{0},\delta_{1})$ we have
\begin{eqnarray}\label{InequalityFirstEigenvalue}
\begin{aligned}
\Delta_f h(x)&=(\Delta_f {\rm Ric}) (V(x),\overline{ V(x)})=\Delta_f R_{i\overline{j}}v^i v^{\overline{j}}(x)=(-R_{i\overline{j}k\overline{l}}R_{\overline{k}l}(x)+\beta R_{i\overline{j}}(x))V^i V^{\overline{j}}\\
&\leq \beta h(x)-h(x)^{2}/(1+\epsilon).
\end{aligned}
\end{eqnarray}
Now let $\tilde{\Phi}(x)=\eta(\frac{r(x)}{c})h(x)$ for $x\in B(x_{0},\delta_{1})$. It is easy to compute that
\begin{eqnarray}
\begin{aligned}
\Delta_{f}\tilde{\Phi}
&=\eta \Delta_{f}h+\frac{2\eta'}{c}<\nabla r,\nabla h>+\frac{\eta'}{c}h\Delta_{f}r+\frac{\eta'' h}{c^{2}}\\
&\leq \eta(\beta h-\frac{h^{2}}{1+\epsilon})+\frac{2\eta'}{c}<\nabla r,\nabla h>+\frac{\eta'}{c}h\Delta_{f}r+\frac{\eta'' h}{c^{2}}
\end{aligned}
\end{eqnarray}
For $x\in B(x_{0},\delta_{1})$, we have $\tilde{\Phi}(x)=\eta(\frac{r(x)}{c})h(x)\geq \eta(\frac{r(x)}{c})\lambda(x)=\Phi(x)\geq \Phi(x_0)=\tilde{\Phi}(x_0).$ We compute at point $x_0$, then we have
\begin{equation}
0\leq \eta(\beta h-\frac{h^{2}}{1+\epsilon})+\frac{2\eta'}{c}<\nabla r,\nabla h>+\frac{\eta'}{c}h\Delta_{f}r+\frac{\eta'' h}{c^{2}}.
\end{equation}
Since $\nabla \tilde{\Phi}(x_0)=0$, we have $\frac{\nabla h}{h}(x_0)=-\frac{\eta'}{\eta}(x_0)\frac{\nabla r}{c}(x_0).$ Note that $h(x_0)<0$, we obtain
\begin{equation}\label{4.4}
0\geq \eta(\beta -\frac{h}{1+\epsilon})+\frac{\eta'}{c}\Delta_{f}r+\frac{1}{c^{2}}(\eta''-\frac{2(\eta')^{2}}{\eta}).
\end{equation}
We consider two cases, depending on the location of $x_0$.\\
{\bfseries Case (i)}\quad Suppose $r(x_0)<c$, so that $\eta(\frac{r(x)}{c})=1$ in a neighborhood of $x_0$, then $\eta'(x_0)=\eta''(x_0)=0$ and from (\ref{4.4}),
\[\frac{h(x_0)}{1+\epsilon}\geq \beta.\]
In the case of $\beta\geq 0$, it is a contradiction. Hence when $\beta\geq 0$, we obtain $\lambda(x)\geq 0$. \\
{\bfseries Case (ii)}\quad Now suppose $r(x_0)\geq c$ and again consider (\ref{4.4}). Note that we may choose $\eta$ so that $-C_1\leq \eta'\leq 0$ and
\begin{equation}\label{CutoffEstimate}
\eta''-\frac{2(\eta')^{2}}{\eta}\geq -C_1
\end{equation}
for some uniform constant $C_1<\infty$. Since $\eta'<0$, applying Lemma \ref{Perelman8.3} (2) and (\ref{CutoffEstimate}) to (\ref{4.4})
\begin{eqnarray}
\begin{aligned}
\frac{\tilde{\Phi}(x_0)}{1+\epsilon}\geq & \frac{\eta'(\frac{r(x_0)}{c})}{c}(\frac{m-1}{2r_0}-\langle\nabla f,\nabla r\rangle(o)-\frac{\beta}{2} r(x_0)+\frac{2}{3}r_0\max_{\overline{B(o,r_0)}}{\rm Ric})\\
&+\eta \beta-\frac{C_1}{c^{2}},
\end{aligned}
\end{eqnarray}
where $C_{1}$ is independent of $c$. Taking $r_0=1/2$ and $c\geq 2$, since $-C_1\leq \eta'\leq 0$, we obtain for all $x\in B(o, c)$
\begin{eqnarray}\label{4.9}
\begin{aligned}
\frac{\lambda(x)}{1+\epsilon}\geq \frac{\tilde{\Phi}(x_0)}{1+\epsilon}\geq & -\frac{C_1}{c}(m-1+|\nabla f|(o)+\frac{1}{3}\max_{\overline{B(o,1/2)}}{\rm Ric}+\frac{1}{c})\\
&+\beta(\eta -\frac{1}{2}\eta'(\frac{r(x_0)}{c})\frac{r(x_0)}{c}).
\end{aligned}
\end{eqnarray}

When take $c\rightarrow \infty$, then the first term of right hand
side of (\ref{4.9}) tends to $0$. In the case of $\beta\geq 0$, we obtain $\lambda(x_0)\geq 0$, it is a contradiction. Hence when $\beta\geq 0$, we have $\lambda(x)\geq 0$ for any $x\in M$.

In the case of $\beta<0$, we using the same argument in the proof of Theorem 1.5 in the author's paper \cite{Zhs}. For convenience of readers, we also provide the details as following.

We only consider to estimate the term
$\frac{1}{2}\eta'(\frac{r(x_{0})}{c})\frac{r(x_{0})}{c}-\eta(\frac{r(x_{0})}{c})$.
Since $x_{0}\in B(o,(1+b)c)-B(o,c)$, we have $1\leq
\frac{r(x_{0})}{c}<1+b$. Define $h_{\eta}(u)$ by
\begin{center}
$h_{\eta}(u)=\frac{1}{2}\eta'(u)u-\eta(u)$.
\end{center}
So we only estimate $h_{\eta}(u)$ for $u\in [1,1+b]$.

If we replace $\eta$ with nonnegative piecewise linear function
$\theta(u)$ such that
\begin{eqnarray*}
\theta(u)=\begin{cases}
1 & \text{if $u\in [0,1]$},\\
\frac{1+b-u}{b} & \text{if $u\in [1,1+b]$},\\
0 & \text{if $u\in [1+b,\infty)$}
\end{cases}
\end{eqnarray*}
then $h_{\theta}(u)=-\frac{2b+2-u}{2b}$ for $u\in [1,1+b]$. So
$h_{\theta}(u)\geq -1$ for $u\in [2,b]$ and $h_{\theta}(u)\geq
-1-\frac{1}{b}$ for $u \in [1,2]$. For any small positive number
$\delta$, we can obtain a $\mathcal{C}^{\infty}$ cutoff function
$\gamma$ after smooth the linear function $\theta$ such that
$\gamma(u)=\theta(u)$ for $u\in [0,1]\cup[2,b]\cup[1+b,\infty)$ and
$-\frac{1+\delta}{b}\leq \gamma'(u)\leq 0$ for $u\in
[1,2]\cup[b,1+b]$. When $u\in [b,1+b]$, it is easy to get $-\gamma(u)\geq -\frac{1+\delta}{b}$. So when $b$ is large and $\delta\leq \frac{b-1}{b+1}$, we have $h_{\gamma}(u)\geq -1-\frac{1+\delta}{b}$ for $
u\in [1,1+b]$. Let $\eta$ equal
$\gamma$, take $c\rightarrow \infty, \delta\rightarrow 0,
b\rightarrow \infty$, by (\ref{4.9}) we obtain
\[\frac{\tilde{\Phi}(x_0)}{1+\epsilon}\geq \beta.\]
Taking $\epsilon\rightarrow 0$, we obtain
\[\lambda(x)\geq \beta\]
for all $x\in M$.

In the case of $\beta\geq 0$, we have proved that ${\rm Ric}(g)\geq 0$. Furthermore, if we assume the orthogonal bisectional curvature is positive, we show that the Ricci curvature is positive. If not, there exists a point $x_0\in M$, and a unit vector $v\in T_{x_0}^{1,0}M$ such that $\lambda(x_0)=\lambda_1(x_0)={\rm Ric}(v,\overline{v})(x_0)=0$. Since the orthogonal bisectional curvature is positive, the scalar curvature $R>0$ (see (3.4) in \cite{Gu-Zhang}. Hence there exists a positive integer $i\in[2,n]$, such that $\lambda_k(x_0)>0$ for all $k\geq i.$ Using the same notation as above, then from equality (\ref{8}), we can rewrite the inequality (\ref{InequalityFirstEigenvalue}) as
\begin{equation}
\Delta_f h(x)<\beta h(x)
\end{equation}
for any $x\in B(x_{0},\delta_{1})$. By the strong maximum principle, see Theorem 3.5 in \cite{Gilbarg-Trudinger}, we have $h(x)=0$ for all $x\in B(x_{0},\delta_{1})$. It is a contradiction. Hence ${\rm Ric}(g)>0$.
\end{proof}

\section{Proof of  Theorem \ref{MainThm1} and Theorem \ref{MainThm2}}
In this section, we prove the main theorems of this paper.
\begin{theorem}[Theorem \ref{MainThm1}]\label{compactnessresult}
Suppose $(M^n, g, f)$ is a complete  shrinking gradient K\"ahler-Ricci soliton , if we assume the orthogonal bisectional curvature is nonnegative and the Ricci curvature is positive, then it must be compact.
\end{theorem}
\begin{proof}
We can use the same argument of Munteanu and Wang \cite{Munteanu-Wang} (also see \cite{Wu-Zhang}) to obtain the theorem. For convenience of readers, we provide the outline of the proof here.

Assume $(M,g,f)$ is noncompact. Let $p$ be a fixed point such that $\nabla f(p)=0$. Denote $\lambda(x)$ as the minimal eigenvalue of the Ricci curvature at $x$. From the proof of Proposition \ref{RicNonnegative}, we know $\lambda$ satisfies the following differential inequality in the sense of barrier, see  (\ref{InequalityFirstEigenvalue}),
\begin{equation*}
\Delta_f \lambda\leq \lambda.
\end{equation*}
Choose a geodesic ball $B(p,r)$ of radius $r$ large enough, then $a:=\min_{\partial B(p,r)}\lambda>0$. We define
\[U:=\lambda-\frac{a}{f}-\frac{2na}{f^{2}},\]
was introduced by Chow-Lu-Yang \cite{Chow-Lu-Yang}. Then if $r$ is large enough, $U>0$ on $\partial B(p,r)$ and $\Delta_{f}U\leq U$ on $M\backslash B(p,r)$. By the maximum principle, it is easy to get $U\geq 0$ on $M\backslash B(p,r)$. That implies ${\rm Ric}\geq \frac{a}{f}$, then using (2) in Lemma 2.1 to obtain $R\geq b\log f$ for some $b>0$.
Then use Lemma 2.1, Lemma 2.2 and Bishop-Gromov volume comparison theorem, for $r$ large enough, we get
\[\int_{B(p,r)}R\geq b\cdot c(n)\log(r){\rm Vol}(B(p,r)).\]

On the other hand, Cao-Zhou \cite{Cao-Zhou} proved that
\[\int_{B(p,r)}R\leq c_{1}(n) {\rm Vol}(B(p,r)).\]
Here $c(n), c_{1}(n)$ are uniform constants depending only on $n$. It is a contradiction.
\end{proof}
By Proposition \ref{RicNonnegative} and Theorem \ref{MainThm1}, we obtain Theorem \ref{MainThm2}.
\begin{theorem}[Theorem \ref{MainThm2}]
Let $(M^{n},g,f)$ be a complete shrinking gradient K\"ahler-Ricci soliton with nonnegative orthogonal bisectional curvature, then we have:
\begin{itemize}
\item[(i)] if the orthogonal bisectional curvature is positive, then $M$ must be compact and isometric-biholomorphic to $\mathbb{C}P^{n}$;
\item[(ii)] if $M$ has nonnegative orthogonal bisectional curvature , then its universal cover $\tilde{M}$ splits as $\tilde{M}=N_{1}\times N_{2}\times \cdots \times N_{l}\times \mathbb{C}^{k}$ isometric-biholomorphically, where $N_i$ are compact irreducible Hermitian Symmetric Spaces.
\end{itemize}
\end{theorem}
\begin{proof}
Part (i). By Proposition \ref{RicNonnegative} and Theorem \ref{compactnessresult}, we know $M$ is compact. Due to Gu-Zhang \cite{Gu-Zhang} (see also Chen \cite{Chen} and Feng-Liu-Wan \cite{Feng-Liu-Wan}), $M$ is biholomorphic to $\mathbb{C}P^{n}$. It exists a Fubini-Study metric $\omega_{FS}$(a special K\"ahler-Einstein metric). Since the complete shrinking gradient K\"ahler-Ricci soliton is a special solution of Ricci flow, applying Tian-Zhu's main theorem in \cite{Tian-Zhu1} (see also \cite{Tian-Zhu2} and \cite{TZZZ}), we know $(M,g)$ is isometric-biholomorphic to $\mathbb{C}P^{n}$.\\

Part (ii). From the proof of the Proposition \ref{RicNonnegative}, we know the tangent bundle $TM$ has an orthogonal decomposition $V_1\oplus V_2$, where $V_1$ and $V_2$ are invariant under parallel translation. Then applying the standard de Rham decomposition theorem (or Hamilton's Lemma in P.176 \cite{Hamilton86}), Gu-Zhang's main theorem (Theorem 1.3 in \cite{Gu-Zhang}) and Part (i) of Theorem 4.2.
\end{proof}
\section*{Acknowledgements}
The author thanks the referees for their helpful suggestions.


\begin{thebibliography}{}
\bibitem{Cao-Hamilton} H. D. Cao, R. Hamilton, Unpublished work.
\bibitem{Cao-Zhou}H. D. Cao, D. T. Zhou, {\em  On complete gradient shrinking Ricci solitons}, J. Differential Geom. 85 (2010), no. 2, 175-185.
\bibitem{Chen} X. X. Chen, {\em On K\"ahler manifolds with positive orthogonal bisectional curvature}, Adv. Math. 215 (2007), 427-445.
\bibitem{Chow} B. Chow,  {\em Expository notes on gradient Ricci solitons}, Unpublished.
\bibitem{Chow-Lu-Yang}B. Chow, P. Lu and B. Yang, {\em  Lower bounds for the scalar curvatures of noncompact gradient Ricci solitons},  C. R. Math. Acad. Sci. Paris 349 (2011), no. 23-24, 1265-1267.
\bibitem{CLN} B. Chow, P. Lu and L. Ni, {\em Hamilton's Ricci flow}, Graduate Studies in Mathematics, Vol. 77, American Mathematical Society, Science Press, 2006.
\bibitem{Feng-Liu-Wan} H. T. Feng, K. F. Liu and X. Y. Wan, {\em Compact K\"ahler manifolds with positive bisectional curvature}, Math. Res. Lett. 24 (2017), no.3, 767-780.
\bibitem{Fang-Man-Zhang} F. Q. Fang, J. W. Man and Z. L. Zhang, {\em Complete gradient shrinking Ricci solitons have finite topological type}, C. R. Math. Acad. Sci. Paris 346 (2008), no. 11-12, 653-656.
\bibitem{Gilbarg-Trudinger} D. Gilbarg and N. S. Trudinger, {\em Elliptic Partial Differential Equations of Second Order}. Springer-Verlag Berlin Heidelberg 2001.
\bibitem{Gu-Zhang}  H. L. Gu and Z. H. Zhang, {\em An extension of Mok's theorem on the gererlized Frankel conjecture}, Sci. China Math. 53 (2010), no. 5, 1253-1264.
\bibitem{Hamilton86} R. S. Hamilton, {\em Four-manifolds with positive curvature operator}, J. Diff. Goem. 24 (1986), no.2, 153-179.
\bibitem{Hamilton} R. S. Hamilton, {\em The formulation of singularities
in the Ricci flow}. \textit{Surveys in Differential Geometry}, Vol. II
(Cambridge, MA, 1993), 7-136, Internat. Press, Cambridge, MA, 1995.
\bibitem{Haslhofer-Muller} R. Haslhofer and R. M\"{u}ller, {\em A compactness theorem for complete Ricci shrinkers}, Geom. Funct. Anal. 21 (2011), no.5, 1091-116.
\bibitem{Munteanu-Wang}O. Munteanu and J. P. Wang,  {\em Positively curved shrinking Ricci solitons are compact}, arxiv: 1504.07898.
\bibitem{Ni}L. Ni, {\em Ancient solutions to K\"ahler-Ricci flow},  Math. Res. Lett. 12 (2005), no. 5-6, 633-653.
\bibitem{Perelman} G. Perelman, {\em The entropy formula for the Ricci
flow and its geometric applications}, Arxiv:math.DG/0211159.
\bibitem{TZZZ} G. Tian, S. J. Zhang, Z. L. Zhang and X. H. Zhu, {\em Perelman's entropy and K\"ahler-Ricci flow on a Fano manifold}, Trans. Amer. Math. Soc. 365(2013), 6669-6695.
\bibitem{Tian-Zhu1} G. Tian and X. H. Zhu, {\em Convergence of K\"ahler-Ricci flow}, J. Amer. Math. Soc., 20(2007) 675-699.
\bibitem{Tian-Zhu2} G. Tian and X. H. Zhu, {\em Convergence of the K\"ahler-Ricci flow on Fano manifolds}, J. Reine Angew. Math. 678(2013), 223-245.
\bibitem{Wu-Zhang} G. Q. Wu and S. J. Zhang, {\em Remarks on shrinking gradient K\"ahler-Ricci solitons with positive bisectional curvature}, C. R. Math. Acad. Sci. Paris, 354(2016), No. 7, 713-716.
\bibitem{Zhz} Z. H. Zhang, {\em On the completeness of gradient Ricci solitons}, Proc. AMS, 137(2009), no. 8, 2755-2759.
\bibitem{Zhs} S. J. Zhang, {\em On a Sharp Volume Estimate for Gradient Ricci Solitons with Scalar Curvature Bounded Below}, Acta Math. Sinica, English Series, 27 (2011), no. 5, 871-882.
\end{thebibliography}
\end{document}